\documentclass[11pt]{amsart}

\usepackage[T1]{fontenc}
\usepackage{lmodern}
\usepackage{microtype}
\usepackage{mathtools,amssymb,amsthm}
\usepackage[margin=1in]{geometry}
\usepackage[hidelinks]{hyperref}

\numberwithin{equation}{section}
\newtheorem{theorem}{Theorem}[section]
\newtheorem{proposition}[theorem]{Proposition}
\newtheorem{lemma}[theorem]{Lemma}
\newtheorem{corollary}[theorem]{Corollary}
\theoremstyle{remark}
\newtheorem{remark}[theorem]{Remark}

\newcommand{\cB}{\mathcal B}
\newcommand{\Sub}{\operatorname{Sub}}
\newcommand{\gen}[1]{\left\langle #1\right\rangle}
\newcommand{\F}{\mathbb F}
\newcommand{\Z}{\mathbb Z}
\newcommand{\Cyc}[1]{(\Z/#1\Z)}
\newcommand{\qbinom}[2]{\genfrac{[}{]}{0pt}{}{#1}{#2}_{p}}

\title[Sharp bounds for subgroup intervals]
  {Sharp bounds for intervals in finite subgroup lattices}
\author{Sebastien Palcoux}
\address{S.~Palcoux, Beijing Institute of Mathematical Sciences and
Applications, Huairou District, Beijing 101408, China}
\email{sebastienpalcoux@gmail.com}
\author{Pablo Spiga}
\address{P.~Spiga, Dipartimento di Matematica e Applicazioni,
Universit\`a degli Studi di Milano-Bicocca, Milano, Italy}
\email{pablo.spiga@unimib.it}
\date{September 9, 2026}


\begin{document}

\begin{abstract}
Let $H\leq G$ be finite and put $n=[G:H]>1$.  If $p$ is the least prime
divisor of $n$, we prove
\[
 |\{K:H\leq K\leq G\}|<c(p)n^{\frac14\log_p n},
\]
where
\[
 c(p)=\prod_{j\geq1}(1-p^{-j})^{-1}\sum_{z\in\Z}p^{-z^2}.
\]
The exponent and the constant are simultaneously optimal, already for
elementary abelian $p$-groups of even rank.  In particular, for arbitrary
$n>1$ we obtain the sharp universal estimate
\[
 |\{K:H\leq K\leq G\}|<7.371968802\,n^{\frac14\log_2 n}.
\]
This removes the polynomial factor from the best previously known relative
bound and extends the sharp absolute subgroup-counting theorem to arbitrary
subgroup intervals.  The group-theoretic argument is elementary.  Its main ingredient is a
weighted packing inequality for relative generating tuples, obtained from an
escape tree on the coset space $G/H$.
\end{abstract}

\maketitle
\section{Introduction}

For an inclusion $H\leq G$ of finite groups, write
\[
 \Sub(G,H)=\{K:H\leq K\leq G\},
 \qquad n=[G:H].
\]
We count the closed interval, so both $H$ and $G$ are included.

The usual generating-set argument (see the opening paragraph of
\cite{Spiga}) adapts to this setting as follows.  Starting with $H$, adjoin
an element of $K$ outside the subgroup already generated until reaching
$K$.  Every step at least doubles the index over $H$, so at most
$d=\lfloor\log_2n\rfloor$ steps are needed.  Pad with the identity coset to
obtain a $d$-tuple in $(G/H)^d$.  Each tuple determines at most one subgroup,
and every intermediate subgroup arises in this way.  Consequently,
\[
 |\Sub(G,H)|\leq n^d\leq n^{\log_2n}.
\]
Here and below, $G/H$ denotes the set of left cosets; $H$ need not be normal.

Borovik, Pyber and Shalev~\cite{BPS} proved the asymptotic estimate
\[
 |\Sub(G)|\leq |G|^{(1/4+o(1))\log_2|G|}.
\]
Spiga~\cite{Spiga} obtained an explicit elementary bound with a polynomial
loss.  Fusari and Spiga~\cite{FS} subsequently removed that loss and proved
\[
 |\Sub(G)|<c(2)|G|^{\frac14\log_2|G|},
\]
with the optimal constant $c(2)=7.371968801\ldots$.

The relative problem was treated recently by Guerra, Mastrogiacomo and
Spiga~\cite{GMS}.  They proved
\[
 |\Sub(G,H)|\leq 7.3722\,
 n^{\frac14\log_2n+1.8919}
\]
and explicitly suggested that the additive constant $1.8919$ in the exponent
should be removable.  Our first theorem confirms this in optimal form and
retains the finer dependence on the least prime divisor of the index.

This removes the polynomial loss, but it does not establish a bound of the
form $n^{\frac14\Omega(n)+c}$ with an absolute constant $c$, where
$\Omega(n)$ counts prime factors with multiplicity.  Such a strengthening
is suggested by the discussion surrounding \cite[Conjecture~1.3]{GMS}.

For a prime $p$ and an integer $r\geq0$, set
\[
 \gamma_{p,0}=1,
 \qquad
 \gamma_{p,r}=\prod_{j=1}^{r}(1-p^{-j})\quad(r\geq1),
\]
and define
\[
 C_p=\prod_{j=1}^{\infty}(1-p^{-j})^{-1},
 \qquad
 \Theta_p=\sum_{z\in\Z}p^{-z^2},
 \qquad
 c(p)=C_p\Theta_p.
\]

\begin{theorem}\label{thm:main}
Let $H\leq G$ be finite, let $n=[G:H]>1$, and let $p$ be the least prime
divisor of $n$.  Put $L=\log_p n$.  Then
\begin{equation}\label{eq:finite-main}
 |\Sub(G,H)|
 \leq
 U_p(n):=
 \sum_{r=0}^{\lfloor L\rfloor}
 \frac{p^{r(L-r)}}{\gamma_{p,r}}
 <c(p)n^{\frac14\log_p n}.
\end{equation}
Both the exponent $1/4$ and the constant $c(p)$ are optimal among inclusions
whose index has least prime divisor $p$.  More precisely, for positive integers $a$, put
\[
 H_a=1,
 \qquad
 G_a=\Cyc{p}^{\,2a}.
\]
Then
\begin{equation}\label{eq:sharp-intro}
 \frac{|\Sub(G_a,H_a)|}
 {[G_a:H_a]^{\frac14\log_p[G_a:H_a]}}
 \longrightarrow c(p)\qquad(a\longrightarrow\infty).
\end{equation}
\end{theorem}

Taking $p=2$ in the proof, whether or not $2$ divides $n$, gives the following
universal form.

\begin{corollary}\label{cor:universal}
For every inclusion $H\leq G$ of finite groups with $n=[G:H]$,
\[
 |\Sub(G,H)|<c(2)n^{\frac14\log_2n}
 \quad\text{if }n>1,
\]
where
\[
 7.371968801461<c(2)=
 \left(\prod_{j\geq1}(1-2^{-j})^{-1}\right)
 \left(\sum_{z\in\Z}2^{-z^2}\right)
 <7.371968801462.
\]
No smaller universal constant is possible.
\end{corollary}

For $H=1$, this is the theorem of Fusari and Spiga~\cite{FS}.
The constant and its optimality are already present in their result;
the assertion here extends the same bound to arbitrary subgroup intervals.

For every integer $r\geq0$, let
\[
 \cB_r^{(p)}(G,H)=
 \{K\in\Sub(G,H):p^r\leq[K:H]<p^{r+1}\}.
\]

\begin{theorem}\label{thm:moment}
Under the hypotheses of Theorem~\ref{thm:main}, one has
\begin{equation}\label{eq:moment}
 \gamma_{p,r}
 \sum_{K\in\cB_r^{(p)}(G,H)}[K:H]^r
 \leq n^r
\end{equation}
for every integer $r\geq0$.
\end{theorem}

The proof of Theorem~\ref{thm:moment} uses only Lagrange's theorem and a
counting lemma for rooted trees.
Regard the interval as a closure system on the coset set $G/H$: a tuple of
cosets closes to the subgroup it generates together with $H$.  Every strict
closure enlargement has multiplicative growth at least $p$.  For each target
subgroup $K$ in the $r$th band, an escape tree produces at least
$\gamma_{p,r}[K:H]^r$ tuples closing to $K$.  The sets of tuples belonging to
distinct targets are disjoint inside $(G/H)^r$, which gives
\eqref{eq:moment}.  Elementary abelian $p$-groups are extremal because every
strict step can have the smallest possible growth.

Pak~\cite[Theorem~1.1.7(i)]{Pak} compares random generation in a finite
group with generation in an elementary abelian $2$-group of rank
$\lceil\log_2|G|\rceil$.  Here we need a relative estimate for a possibly
nonnormal subgroup, at tuple length $\lfloor\log_p[K:H]\rfloor$.
This is the length required by the packing argument.

\section{Relative generation by escaping cosets}

We begin with two elementary observations.  Throughout this section, $H\leq
K\leq G$ are finite and cosets are written as $xH$.  We use $K/H$ only as a
set; no normality assumption is made.

\begin{lemma}\label{lem:index-growth}
Let $n=[G:H]>1$ and let $p$ be its least prime divisor.  If
$H\leq A<B\leq G$, then $[B:A]\geq p$.
\end{lemma}

\begin{proof}
The factorization
\[
 n=[G:B][B:A][A:H]
\]
shows that $[B:A]$ divides $n$.  Since it is larger than one, it is at least
the least prime divisor of $n$.
\end{proof}

\begin{lemma}\label{lem:tree}
Let $r\geq0$ be an integer and let $\mathcal T$ be a finite rooted tree
all of whose leaves have depth $r$.
For a vertex $v$ of depth less than $r$, let $b(v)$ be its number of children.
Then
\[
 |\operatorname{Leaf}(\mathcal T)|
 \geq
 \min_{P}\prod_{v\in P}b(v),
\]
where $P$ ranges over root-to-leaf paths and the product ranges over the
nonleaf vertices of $P$.
\end{lemma}

\begin{proof}
We argue by induction on $r$.  The assertion is clear for $r=0$.  Suppose the
root has $b$ children and the corresponding subtrees have leaf counts
$N_1,\ldots,N_b$.  By induction, $N_i\geq q_i$, where $q_i$ is the minimum
degree product along a path within the $i$th subtree.  Hence
\[
 |\operatorname{Leaf}(\mathcal T)|
 =\sum_{i=1}^{b}N_i
 \geq\sum_{i=1}^{b}q_i
 \geq b\min_iq_i,
\]
which is precisely the required root-to-leaf minimum.
\end{proof}

For an $r$-tuple $\mathbf x=(x_1H,\ldots,x_rH)$ of cosets in $K/H$, put
\[
 \gen{H,\mathbf x}=\gen{H,x_1,\ldots,x_r}.
\]
This does not depend on the chosen representatives: replacing $x_i$ by
$x_ih_i$, with $h_i\in H$, leaves the generated subgroup unchanged.

\begin{proposition}\label{prop:escape}
Let $n=[G:H]>1$, let $p$ be its least prime divisor, let $r\geq0$ be an
integer, and suppose
\[
 p^r\leq m:=[K:H]<p^{r+1}.
\]
Then the number of tuples $\mathbf x\in(K/H)^r$ satisfying
$\gen{H,\mathbf x}=K$ is at least
\[
 \gamma_{p,r}m^r.
\]
\end{proposition}

\begin{proof}
We construct a rooted tree of depth $r$.  The vertices at depth $i$ are
certain prefixes $(x_1H,\ldots,x_iH)$, with associated subgroup
\[
 K_i=\gen{H,x_1,\ldots,x_i};
 \qquad K_0=H.
\]
If $K_i<K$, attach one child for each coset $xH\in K/H$ not contained in
$K_i$, and associate to that child the subgroup $\gen{K_i,x}$.  If $K_i=K$,
attach one child for every coset in $K/H$, keeping the associated subgroup
equal to $K$.

The construction is independent of representatives.  Indeed, since
$H\leq K_i$, the coset $xH$ is either contained in $K_i$ or disjoint from it,
and $\gen{K_i,xh}=\gen{K_i,x}$ for every $h\in H$.  At a proper node with
associated subgroup $M<K$, exactly $[M:H]$ of the $m$ cosets lie in $M$.
Thus its branching degree is
\begin{equation}\label{eq:branching}
 m-[M:H]
 =m\left(1-\frac1{[K:M]}\right).
\end{equation}
A node associated with $K$ has branching degree $m$.

Every leaf generates $K$.  Otherwise all $r$ moves along its path would be
strict.  Lemma~\ref{lem:index-growth} would then give
\[
 [K_r:H]\geq p^r.
\]
On the other hand, $K_r<K$ would imply
\[
 [K_r:H]\leq m/p<p^r,
\]
a contradiction.

Consider a path that first reaches $K$ after $d\leq r$ strict moves.  For
$0\leq i\leq d$, put $t_i=[K:K_i]$, so that $t_d=1$.  By
Lemma~\ref{lem:index-growth},
\[
 t_i=[K_{i+1}:K_i]t_{i+1}\geq pt_{i+1}.
\]
Consequently $t_{d-j}\geq p^j$ for $1\leq j\leq d$.  Dividing the product of
the branching degrees on the path by $m^r$ and using
\eqref{eq:branching}, we obtain
\[
 \prod_{i=0}^{d-1}(1-t_i^{-1})
 \geq\prod_{j=1}^{d}(1-p^{-j})
 =\gamma_{p,d}
 \geq\gamma_{p,r}.
\]
The remaining $r-d$ levels contribute the normalized factor one.  Every
root-to-leaf degree product is therefore at least $\gamma_{p,r}m^r$.
Lemma~\ref{lem:tree} gives at least that many leaves, and each leaf is a
distinct generating tuple.
\end{proof}

\begin{remark}\label{rem:equality-escape}
Write $\F_p$ for the finite field with $p$ elements.
The constant in Proposition~\ref{prop:escape} is best possible.  For
$H=1$ and $K=\Cyc{p}^{\,r}$, the escaping tuples are precisely the ordered bases of
$\F_p^r$, and their number is
\[
 (p^r-1)(p^r-p)\cdots(p^r-p^{r-1})
 =\gamma_{p,r}|K|^r.
\]
\end{remark}

\section{Packing the relative generating tuples}

We now prove the bandwise estimate and derive the closed formulas in
Theorem~\ref{thm:main}.

\begin{proof}[Proof of Theorem~\ref{thm:moment}]
Fix an integer $r\geq0$ and, for every $K\in\cB_r^{(p)}(G,H)$, let $E_r(K)$ be the set of all
$r$-tuples of cosets in $(K/H)^r$ which generate $K$ together with $H$.
Proposition~\ref{prop:escape} gives
\begin{equation}\label{eq:E-size}
 |E_r(K)|\geq\gamma_{p,r}[K:H]^r.
\end{equation}
Each $E_r(K)$ is a subset of $(G/H)^r$.  Moreover, these subsets are pairwise
disjoint: a tuple $(x_1H,\ldots,x_rH)$ determines the subgroup
$\gen{H,x_1,\ldots,x_r}$ uniquely.  Since $|(G/H)^r|=n^r$, summing
\eqref{eq:E-size} gives
\[
 \gamma_{p,r}\sum_{K\in\cB_r^{(p)}(G,H)}[K:H]^r
 \leq\sum_{K\in\cB_r^{(p)}(G,H)}|E_r(K)|
 \leq n^r.\qedhere
\]
\end{proof}

The moment inequality immediately yields both bandwise and individual-layer
estimates.

\begin{corollary}\label{cor:band}
For every integer $r\geq0$,
\begin{equation}\label{eq:band-card}
 |\cB_r^{(p)}(G,H)|
 \leq\gamma_{p,r}^{-1}p^{r(\log_pn-r)}.
\end{equation}
More precisely, if an integer $m$ satisfies $p^r\leq m<p^{r+1}$, then
\begin{equation}\label{eq:layer-card}
 |\{K\in\Sub(G,H):[K:H]=m\}|
 \leq\gamma_{p,r}^{-1}(n/m)^r.
\end{equation}
\end{corollary}

\begin{proof}
Since $[K:H]^r\geq p^{r^2}$ for $K\in\cB_r^{(p)}(G,H)$,
\eqref{eq:moment} gives \eqref{eq:band-card}.
For the fixed-index assertion, put
\[
 N_m=|\{K\in\Sub(G,H):[K:H]=m\}|.
\]
These $N_m$ subgroups all lie in $\cB_r^{(p)}(G,H)$, and each contributes
$m^r$ to the sum.  Discarding the other nonnegative terms in
\eqref{eq:moment} therefore gives
\[
 \gamma_{p,r}N_m m^r
 \leq\gamma_{p,r}\sum_{K\in\cB_r^{(p)}(G,H)}[K:H]^r
 \leq n^r.
\]
Division by $\gamma_{p,r}m^r>0$ proves \eqref{eq:layer-card}.
\end{proof}

The bands with $0\leq r\leq\lfloor\log_p n\rfloor$ partition
$\Sub(G,H)$.  Notice that the zeroth band is $\{H\}$.  Summing
\eqref{eq:band-card} proves the first inequality of
\eqref{eq:finite-main}.  It remains to evaluate the resulting discrete
Gaussian sum.

\begin{lemma}\label{lem:theta}
For every real $x$ and every $p>1$,
\[
 \sum_{r\in\Z}p^{-(r-x)^2}
 \leq\sum_{z\in\Z}p^{-z^2}.
\]
Equality holds if and only if $x\in\Z$.
\end{lemma}

\begin{proof}
Poisson summation applied to the Gaussian gives
\[
 \sum_{r\in\Z}e^{-(\log p)(r-x)^2}
 =\sqrt{\frac{\pi}{\log p}}
 \sum_{\ell\in\Z}
 e^{-\pi^2\ell^2/\log p}e^{2\pi i\ell x}.
\]
Taking real parts, every nonconstant Fourier coefficient is strictly positive.
The right-hand side is therefore maximized exactly when
$\cos(2\pi\ell x)=1$ for every $\ell$, equivalently when $x\in\Z$.
\end{proof}

\begin{proof}[Completion of the proof of Theorem~\ref{thm:main}]
Since $\gamma_{p,r}^{-1}<C_p$ for every finite $r$, completing the square gives
\begin{align*}
 U_p(n)
 &<C_p\sum_{r=0}^{\lfloor L\rfloor}p^{r(L-r)}\\
 &\leq C_p p^{L^2/4}
 \sum_{r\in\Z}p^{-(r-L/2)^2}\\
 &\leq C_p\Theta_p p^{L^2/4}
 =c(p)n^{\frac14\log_pn},
\end{align*}
where the last inequality is Lemma~\ref{lem:theta}.  The asserted sharpness
will be proved in Section~\ref{sec:sharpness}.
\end{proof}

\begin{proof}[Proof of Corollary~\ref{cor:universal}]
The preceding argument only requires a lower bound for the index of every
strict step.  Such an index is always at least two, so the proof applies with
$p=2$ for every $n>1$.  Optimality follows from the elementary abelian
$2$-groups considered in Section~\ref{sec:sharpness}.
\end{proof}

As in \cite[Corollary~1.2]{GMS}, the subgroup bound has a direct
interpretation for transitive permutation groups.

\begin{corollary}\label{cor:partitions}
Let a finite group $G$ act transitively on a finite set $X$ of size $n>1$, and let $p$ be
the least prime divisor of $n$.  The number of $G$-invariant partitions
of $X$, including the two trivial partitions, is less than
$c(p)n^{\frac14\log_p n}$.  In particular it is less than
$c(2)n^{\frac14\log_2 n}$, with the same optimal universal constant.
\end{corollary}

\begin{proof}
Fix $x\in X$ and put $H=G_x$.  The usual correspondence sends
$H\leq K\leq G$ to the partition into translates of the block $Kx$;
its inverse sends the block containing $x$ to its setwise stabilizer.
Apply Theorem~\ref{thm:main} and Corollary~\ref{cor:universal}.
Regular actions of $\Cyc{2}^{\,2a}$ give optimality.
\end{proof}

\section{Sharpness and Gaussian binomial asymptotics}\label{sec:sharpness}

Recall that the number of $k$-dimensional subspaces of $\F_p^d$ is
\[
 \qbinom{d}{k}
 =\prod_{i=0}^{k-1}\frac{p^{d-i}-1}{p^{k-i}-1}.
\]
These standard subspace counts underlie the optimality argument of
Fusari and Spiga~\cite{FS}.  We give the short limiting calculation for
general $p$ to make the precise constant self-contained.

\begin{proposition}\label{prop:gaussian-asymptotic}
For every prime $p$,
\[
 p^{-a^2}\sum_{k=0}^{2a}\qbinom{2a}{k}
 \longrightarrow c(p)
 \qquad(a\longrightarrow\infty).
\]
\end{proposition}

\begin{proof}
Write $k=a+j$.  The Gaussian product formula gives
\begin{equation}\label{eq:central-gaussian}
 \qbinom{2a}{a+j}
 =p^{a^2-j^2}
 \frac{\displaystyle\prod_{u=a-j+1}^{2a}(1-p^{-u})}
 {\displaystyle\prod_{u=1}^{a+j}(1-p^{-u})}.
\end{equation}
For fixed $j\in\Z$, the numerator tends to one and the denominator tends to
$C_p^{-1}$.  Hence
\[
 p^{-a^2}\qbinom{2a}{a+j}\longrightarrow C_pp^{-j^2}.
\]
Define the left-hand side to be zero when $|j|>a$.  The numerator in
\eqref{eq:central-gaussian} is at most one, while its denominator is at least
$C_p^{-1}$.  Thus the terms are bounded by the summable majorant
$C_pp^{-j^2}$.  Dominated convergence on $\Z$ gives
\[
 \lim_{a\to\infty}
 p^{-a^2}\sum_{k=0}^{2a}\qbinom{2a}{k}
 =C_p\sum_{j\in\Z}p^{-j^2}=c(p).\qedhere
\]
\end{proof}

\begin{proof}[Proof of the sharpness assertions in Theorem~\ref{thm:main}]
For $H_a=1$ and $G_a=\Cyc{p}^{\,2a}$, the interval
$\Sub(G_a,H_a)$ is the subspace lattice of $\F_p^{2a}$.  Therefore
\[
 |\Sub(G_a,H_a)|=\sum_{k=0}^{2a}\qbinom{2a}{k}.
\]
Here $n_a=[G_a:H_a]=p^{2a}$ and
\[
 n_a^{\frac14\log_pn_a}=p^{a^2}.
\]
The limit in \eqref{eq:sharp-intro} follows from
Proposition~\ref{prop:gaussian-asymptotic}.  In particular, neither the
coefficient $1/4$ nor the multiplicative constant $c(p)$ can be decreased.\qedhere

\end{proof}

\begin{remark}
Odd-dimensional elementary abelian groups produce the half-shifted theta sum
$\sum_{j\in\Z}p^{-(j+1/2)^2}$, which is strictly smaller than $\Theta_p$ by
Lemma~\ref{lem:theta}.  Thus even rank is the extremal parity for the universal
constant.
\end{remark}

\begin{remark}\label{rem:primepower}
At index $p^a$, the exact maximum is the number
$\sum_{k=0}^{a}\qbinom{a}{k}$ of subspaces of $\F_p^a$.
This is a consequence of the classical Gaussian layer bound for
$p$-groups; see Shalev~\cite[Corollary~4.2]{Shalev} and
Spiga~\cite[Remark~3.1]{Spiga}.  The relative $p$-group layer estimate is recorded in Guerra, Mastrogiacomo
and Spiga~\cite[Proposition~2.1]{GMS}.  To pass to arbitrary inclusions of
index $p^a$, choose a Sylow $p$-subgroup $P$ of $G$ containing a Sylow
$p$-subgroup of $H$.  Then $G=PH$ as sets, and
$K=(K\cap P)H$ for every $H\leq K\leq G$.  Thus
$K\mapsto K\cap P$ is an injection preserving the index
$[G:K]=[P:K\cap P]$.
Equality holds when $H$ is normal and $G/H\cong\Cyc{p}^{\,a}$.
\end{remark}

\section*{Acknowledgements}
The first author was supported by the National Natural Science Foundation
of China (NSFC), Grant no.~12471031.

\section*{AI assistance and formal verification}
This manuscript was prepared with assistance from GPT-5.6 Sol and
GPT-6 Astra.  The authors take full responsibility for the final text.
Formal proofs of the main results have been checked in Lean and registered
with Palomar as \texttt{PALOMAR-2026-09-03-000005}, version~1~\cite{Palomar}.


\begin{thebibliography}{99}

\bibitem{BPS}
A.~V. Borovik, L.~Pyber and A.~Shalev,
\emph{Maximal subgroups in finite and profinite groups},
Trans. Amer. Math. Soc. \textbf{348} (1996), no.~9, 3745--3761.
\href{https://doi.org/10.1090/S0002-9947-96-01665-0}
{doi:10.1090/S0002-9947-96-01665-0}.

\bibitem{FS}
M.~Fusari and P.~Spiga,
\emph{On the maximum number of subgroups of a finite group},
J. Algebra \textbf{635} (2023), 486--526.
\href{https://doi.org/10.1016/j.jalgebra.2023.07.047}
{doi:10.1016/j.jalgebra.2023.07.047}.

\bibitem{GMS}
L.~Guerra, F.~Mastrogiacomo and P.~Spiga,
\emph{Counting subgroups of a finite group containing a prescribed subgroup},
J. Pure Appl. Algebra \textbf{229} (2025), no.~9, Article 108038.
\href{https://doi.org/10.1016/j.jpaa.2025.108038}
{doi:10.1016/j.jpaa.2025.108038}.

\bibitem{Pak}
I.~Pak,
\emph{What do we know about the product replacement algorithm?},
in: Groups and Computation III (Columbus, OH, 1999),
Ohio State Univ. Math. Res. Inst. Publ., vol.~8,
de Gruyter, Berlin, 2001, pp.~301--347.
\href{https://doi.org/10.1515/9783110872743.301}
{doi:10.1515/9783110872743.301}.

\bibitem{Shalev}
A.~Shalev,
\emph{Growth functions, $p$-adic analytic groups, and groups of finite coclass},
J. London Math. Soc. (2) \textbf{46} (1992), no.~1, 111--122.
\href{https://doi.org/10.1112/jlms/s2-46.1.111}
{doi:10.1112/jlms/s2-46.1.111}.

\bibitem{Spiga}
P.~Spiga,
\emph{An explicit upper bound on the number of subgroups of a finite group},
J. Pure Appl. Algebra \textbf{227} (2023), no.~6, Article 107312.
\href{https://doi.org/10.1016/j.jpaa.2022.107312}
{doi:10.1016/j.jpaa.2022.107312}.

\bibitem{Palomar}
Palomar Registry,
\emph{Sharp bounds for intervals in finite subgroup lattices},
record \texttt{PALOMAR-2026-09-03-000005}, version~1, September~3, 2026.
\href{https://palomar-registry.org/entry?id=PALOMAR-2026-09-03-000005&version=1}
{Immutable registration record}.

\end{thebibliography}
\end{document}